\documentclass{article}

\usepackage{arxiv}
\usepackage[numbers]{natbib}

\usepackage[utf8]{inputenc} 
\usepackage[T1]{fontenc}    
\usepackage{hyperref}       
\usepackage{url}            
\usepackage{booktabs}       
\usepackage{amsfonts}       
\usepackage{nicefrac}       
\usepackage{microtype}      
\usepackage{lipsum}		
\usepackage{graphicx}
\usepackage{natbib}
\usepackage{doi}

\usepackage{graphicx}%
\usepackage{multirow}%
\usepackage{amsmath,amssymb,amsfonts}%
\usepackage{amsthm}%
\usepackage{mathrsfs}%
\usepackage[title]{appendix}%
\usepackage{xcolor}%
\usepackage{textcomp}%
\usepackage{manyfoot}%
\usepackage{booktabs}%
\usepackage{algorithm}%
\usepackage{algorithmicx}%
\usepackage{algpseudocode}%
\usepackage{listings}%
\usepackage{enumerate}

\theoremstyle{thmstyleone}%
\newtheorem{theorem}{Theorem}%
\newtheorem{proposition}[theorem]{Proposition}%
\newtheorem{conjecture}[theorem]{Conjecture}%
\newtheorem{corollary}[theorem]{Corollary}%

\theoremstyle{thmstyletwo}%

\theoremstyle{thmstylethree}%
\newtheorem{definition}{Definition}%

\title{Bounds for the Permutation Flowshop Scheduling Problem: New Framework and Theoretical Insights}

\author{ \href{https://orcid.org/0009-0007-8014-204X}{\includegraphics[scale=0.06]{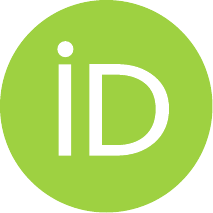}\hspace{1mm}J.A. Alejandro-Soto} \\
        Área de Computación\\
        Centro de Investigación en Matemáticas (CIMAT)\\
        Guanajuato, México\\
	\texttt{jose.alejandro@cimat.mx} \\
	\And
    \href{https://orcid.org/0000-0002-5431-5927}{\includegraphics[scale=0.06]{orcid.pdf}\hspace{1mm}Carlos Segura} \\
	Área de Computación\\
        Centro de Investigación en Matemáticas (CIMAT)\\
        Guanajuato, México\\
	\texttt{carlos.segura@cimat.mx} \\
    \And
	\href{https://orcid.org/0000-0001-9326-7713}{\includegraphics[scale=0.06]{orcid.pdf}\hspace{1mm}Joel Antonio Trejo-Sanchez} \\
	Unidad Mérida\\
	Centro de Investigación en Matemáticas (CIMAT)\\
	Yucatán, México \\
	\texttt{joel.trejo@cimat.mx} \\
}

\renewcommand{\shorttitle}{}

\hypersetup{
pdftitle={A template for the arxiv style},
pdfsubject={q-bio.NC, q-bio.QM},
pdfauthor={David S.~Hippocampus, Elias D.~Striatum},
pdfkeywords={First keyword, Second keyword, More},
}

\begin{document}
\maketitle

\begin{abstract}
In this work, we use the matrix formulation of the Permutation Flowshop Scheduling Problem with makespan minimization 
to derive an upper bound and a general framework for obtaining lower bounds.  
The proposed framework involves solving a min-max or max-min expression over a set of paths.  
We introduce a family of such path sets for which the min-max expression can be solved in polynomial time under certain bounded parameters. 
To validate the proposed approach, we test it on the Taillard and VRF benchmark instances, the two most widely used datasets in PFSP research. Our method improves the bounds in 112 out of the 120 Taillard instances and 430 out of the 480 VRF instances. These improvements include both small and large instances, highlighting the scalability of the proposed methodology. 
Additionally, the upper bound is used to give a more accurate estimate of the number of possible makespan values for a given instance and to present 
asymptotic results which provide advances in a conjecture given by Taillard related to the quality of one of the most popular lower bounds, as well as the asymptotic approximation ratio of any algorithm. 
\end{abstract}


\section{Introduction}

Scheduling is a fundamental area of study in the field of operations research. 
In this work, we focus on the Permutation Flowshop Scheduling Problem (PFSP) with makespan minimization. 
For a comprehensive introduction to scheduling, the reader is referred to~\cite{pinedo-2022}. 
The PFSP consists of $N$ jobs to be processed on $M$ machines arranged in series. 
For convenience, jobs are numbered from $1$ to $N$ and machines from $1$ to $M$. 
Each job must be processed in order from machine $1$ to machine $M$ and each machine must process the jobs in the same order without interruption.  
The processing time of job $j$ on machine $i$ is denoted by $t_{i,j}$, and these times are assumed to be constant and non-negative.  
Thus, in the PFSP, a permutation of the jobs $\pi = \{\pi_1, \dots, \pi_N \}$ represents a complete and valid schedule. 
The objective is to minimize the makespan, defined as the time at which the last machine finishes processing the last job and is denoted as $C_{max}(\pi)$.
Let $C_{i, j}$ be the time at which machine $i$ finishes processing job $j$; then, for a given permutation $\pi$
\begin{small}
\begin{align}
    C_{1, \pi_1} &= t_{1, \pi_1} \notag\\
    C_{i, \pi_1} &= t_{i, \pi_1} + C_{i-1, \pi_1} && \text{For $1 < i \leq M$} \notag\\
    C_{1, \pi_j} &= t_{1, \pi_j} + C_{1, \pi_{j-1}} && \text{For $1 < j \leq N$} \notag\\
    C_{i, \pi_j} &= t_{i, \pi_j} + \max\{C_{i-1, \pi_j}, C_{i,\pi_{j-1}}\}  && \text{otherwise} \label{MakespanFormula}
\end{align}
\end{small}

Hence, the makespan satisfies $C_{max}(\pi) = C_{M, \pi_N}$ and can be computed in $O(NM)$ time for a given permutation $\pi$. 
For an instance $I$ of the PFSP, we denote the minimum makespan in $I$ by $OPT(I)$, or simply $OPT$ when the instance is clear from the context. 

There also exists a matrix formulation of the PFSP based on the notion of critical paths introduced by Johnson~\cite{JohnsonCriticalPath}. 
While the scheduling-based formulation is more common in operations research, this alternative formulation facilitates the analysis of specific PFSP properties and often enables simpler proofs. This formulation is presented in Section~\ref{section:Preliminaries}. 

The PFSP is known to be NP-Hard for more than two machines~\cite{NPHard}; 
consequently, numerous heuristic solvers have been proposed~\cite{ReviewReciente}. 
To evaluate their performance, two benchmark sets are mainly used: the Taillard instances~\cite{TaillardConjecture} and the VRF instances~\cite{InstancesVRF}. 
The VRF set was proposed since the Taillard instances were reaching exhaustion, and consists of $480$ instances divided into two groups, one small and one large, each with $240$ instances.

Despite advances in hardware and methodology, a significant gap remains between the best-known solutions and the best-known lower bounds for the PFSP. 
Moreover, improvements in upper bounds are usually small. 
This suggests that there is considerable room for improvement in the development of tighter lower bounds. 
The matrix formulation of the PFSP has proven effective for deriving simple and scalable approaches~\cite{MakespanDistribution, EfficientlySolvableCases}. 
Using sets of paths over the matrix, we present a new general scheme to derive lower bounds.
In particular, we propose the bounds $LB_{p,s}$, which outperform the current best-known lower bounds for many benchmark instances. 
Moreover, $LB_{p,s}$ are simpler in structure than the state-of-the-art lower bound~\cite{LowerBoundsNew}. 

In contrast, few results have been reported on upper bounds for the PFSP. 
These are important because under the assumption of integer processing times, it is possible to estimate the number of distinct makespan values an instance may attain. 
This assumption is not restrictive, since PFSP instances require concrete processing times. All the processing times can then be multiplied by an appropriate number to convert them to integers. 
An upper bound that depends on the values $N$, $M$, the shortest processing time, and longest processing time is given. 
Although a previous estimate by Heller~\cite{Heller} already indicated that the number of distinct makespan values is usually much smaller than the number of permutations, our proposed bound improves this estimation in certain cases. 
Furthermore, it enables the derivation of asymptotic results concerning the approximation ratio of any algorithm and provides progress toward Taillard's conjecture, supporting its potential validity.
 
The rest of the paper is organized as follows. 
Section \ref{section:Preliminaries} provides a summary of related work and necessary notation. 
Section \ref{sectioni:NewLB} introduces the proposed lower and upper bounds.  
Section \ref{section:AsymptoticBehavior} presents the asymptotic results. 
Section \ref{section:PerformanceLB} evaluates the performance of $LB_{p,s}$ across the benchmark instances. 
Finally, Section~\ref{section:Conclusion} provides concluding remarks and future research directions. 


\section{Preliminaries and related work}
\label{section:Preliminaries}

In this section, the related work and necessary definitions are presented. 

\subsection{Matrix formulation}

The matrix formulation has proven useful for deriving simple and effective methods, which are often easier to prove than their counterparts in the traditional scheduling version~\cite{MakespanDistribution, EfficientlySolvableCases}. 

\begin{definition}
Given a matrix $A$ with $R$ rows and $C$ columns, a path $P$ that only moves right or down 
is called an \textit{RD-path}. That is, from cell $(i,j)$, the path may proceed to either $(i + 1, j)$ or $(i, j + 1)$, whenever those cells exist. When the start and end positions are not specified, we assume that the path goes from the top-left corner to the bottom-right corner. Let $s(\pi, P)$ denote the sum of the entries in $P$ after permuting the columns of $A$ according to a permutation $\pi$, that is, $s(\pi, P) = \sum_{(i,j) \in P} A_{i,\pi_j}$
\end{definition}

\begin{definition}
A \textit{critical path} in a matrix is an RD-path with maximum sum.
\end{definition}

For a fixed permutation $\pi$, the makespan expression in Equation~\eqref{MakespanFormula} coincides with the sum of a critical path in the matrix induced by $\pi$. 
Thus, the PFSP can be reformulated as the problem of finding a column permutation of a matrix that minimizes the sum of a critical path.

\subsection{Bounds}

One of the first lower bounds proposed for the PFSP is 
based on the observation that each machine must process all jobs, yielding $LB_M = \max_{1 \leq i \leq M} \sum_{j=1}^N t_{i,j}$~\cite{Heller}. 
This was later improved~\cite{MachineLBImproved, TaillardConjecture} by incorporating for each machine $i$, the minimum possible time that this machine must wait before starting and after finishing, resulting in $LB_M^+ = \max_{1 \leq i \leq M} \{ \min_{1 \leq j \leq N} \sum_{i' < i} t_{i', j} + \sum_{1 \leq j \leq N} t_{i,j} + \min_{1 \leq j \leq N} \sum_{i' > i} t_{i', j} \}$. 
Another simple lower bound is derived by considering that each job must be processed on all machines: $LB_J = \max_{1 \leq j \leq N} \sum_{1 \leq i \leq M} t_{i,j}$. 
This was improved in~\cite{JobBaseLB_Improved} by observing that for a specific job, all other jobs must be scheduled either before or after the given job. 
Specifically, for each job $j$ we add the minimum between the processing times of all other jobs on the first and last machine, leading to $LB_J^+ = \max_{1 \leq j \leq N} \{ \sum_{1 \leq i \leq M} t_{i,j} + \sum_{j' \neq j} \min \{t_{1,j'}, t_{M, j'}\}\}$. 
All of these bounds can be computed in $O(NM)$ time. 
The lower bounds $LB_J^+$ and $LB_M^+$ are commonly referred to as job-length and machine-load based lower bounds, respectively.  
Taillard considered $LB = \max \{LB_M^+, LB_J\}$ and conjectured the following

\begin{conjecture}[\cite{TaillardConjecture}]
\label{Conj:Taillard}
For a fixed number of machines $M$ 
$$
\lim_{N \to \infty} \mathbb{P}(LB = OPT) = 1
$$
\end{conjecture}

To the best of our knowledge, no progress has been made toward proving or disproving Conjecture~\ref{Conj:Taillard}. 

A general scheme to derive lower bounds is to relax the problem by allowing machines to process jobs simultaneously~\cite{LowerBoundLADHARI}. The bounds presented in~\cite{LowerBoundLADHARI} were also used to guide the design of the VRF instances, where the authors aimed to maximize the gap between the upper bounds generated by state-of-the-art solvers and the known lower bounds.
A recent study~\cite{LowerBoundsNew} analyzed the dominance relationships among several bounds, including  $LB_M^+$, $LB_J^+$, a mixed-integer linear program (MILP) based bound $L_c$~\cite{MILPBound} and an incremental lower bound $\gamma_{i,k}$~\cite{kumarBound} that computes a lower bound for the time at which machine $i$ finishes processing the $k$-th job.  
It was established that $OPT \ge L_c \geq \gamma_{M,N} \geq L_M^+$ and that, in general, these bounds are incomparable with $LB_J^+$. 
The MILP and incremental bounds were then combined into a new bound $L_c^+$, which dominates all previous ones.  
To the best of our knowledge, the $L_c^+$ bound represents the state of the art for the Taillard instances and the small group of VRF instances, while the large group of VRF instances retains the original lower bounds from when the benchmark was proposed.

In contrast to the lower bounds, far less work has been done on the upper bounds. To our knowledge, the only such result is from Heller~\cite{Heller}. 
The proposed upper bound is obtained by processing jobs one after another, letting each job complete its processing on all machines before starting the next one. This results in a schedule of total length $UB_H=\sum_{1 \leq i \leq M} \sum_{1 \leq j \leq N} t_{i,j}$, which is independent of the permutation of jobs.
Assuming integer processing times, the number of distinct makespans can be bounded by the interval length of any lower and upper bound. 
Heller~\cite{Heller} used this to estimate the number of distinct makespans as $UB_H - LB_M + 1$. 

\subsection{Approximation guarantees}

Since the PFSP is NP-hard~\cite{NPHard}, it is important to understand the worst-case performance of heuristic algorithms. 

\begin{definition}
Given an instance $I$ of a minimization problem and an algorithm $A$, the \textit{approximation ratio} of $A$ on $I$ is defined as $\alpha_A(I) = \frac{ A(I) }{ OPT(I) }$
\end{definition}

\begin{definition}
Given a set of instances $\mathcal{I}$ of a minimization problem and an algorithm $A$, the approximation ratio of $A$ over $\mathcal{I}$ is defined as $\alpha_A( \mathcal{I} ) = \sup \{ \alpha_A(I) \mid I \in \mathcal{I} \}$
\end{definition}

To the best of our knowledge, no algorithm with a constant approximation ratio has been presented for the general PFSP, and all currently known ratios depend on $N$ and $M$.  
It is known that unless $P = NP$, no polynomial-time algorithm can achieve an approximation ratio below $\frac{5}{4}$~\cite{WilliamsonApprox}. 
A first result is that any algorithm achieves an approximation ratio of at most $\min\{N, M\}$~\cite{pinedo-2022}. 
The widely used NEH heuristic~\cite{NEH} has an approximation ratio between $\sqrt{\frac{M}{2}}$ and $\frac{M + 1}{2}$~\cite{NEHApprox}. 
Sviridenko~\cite{SviridenkoApprox} proposed a randomized algorithm based on Chernoff bounds with an approximation ratio of $O(\sqrt{M \log M})$. Another approach~\cite{SqrtMinNM-Approx} achieves an approximation ratio of $O(\sqrt{\min\{N, M\}})$ via a connection to the longest increasing subsequence problem. 
This method was later shown to have time complexity $\Omega(N^4 M)$~\cite{LinearApprox}, and a linear-time alternative with approximation ratio $2\sqrt{2N + M}$ was introduced using weighted monotone subsequences. 

A complementary approach focuses on absolute guarantees, which considers the difference of the solution generated by an algorithm and the optimum rather than the ratio. 
Several geometric methods~\cite{Geo1, Geo2, Geo3} yield algorithms whose makespans differ from the optimum by at most $\mu(M)b$, where $\mu(M)$ is a scalar and $b = \max_{1 \leq i \leq M} \max_{1 \leq j \leq N} t_{i,j}$ is the maximum processing  time. 
It has been shown~\citep{SurveyPerformanceGuarantees} that for fixed $M$, the probability that the approximation ratio of such an algorithm with an absolute approximation guarantee of the form $\mu(M)b$ exceeds one converges to zero as $N \to \infty$. 

For further results on approximation guarantees in PFSP, we refer the reader to~\cite{HeuristicWorstCase, SurveyPerformanceGuarantees}. 

Finally, we recall the weak law of large numbers, which will be used to derive asymptotic results in Section~\ref{section:AsymptoticBehavior}. 

\begin{theorem}[Weak Law of Large Numbers~\cite{ProbabilityBook}]
\label{Teo:WLLN}
Let $X_1, X_2, \dots$ be a sequence of independent and identically distributed random variables, each having finite mean $\mathbb{E}X_i = \mu$. Then, for any $\epsilon > 0$,
\begin{align*}
   \lim_{N \to \infty} \mathbb{P} \left(  \left| \frac{1}{N} \sum_{i = 1}^N X_i - \mu  \right| > \epsilon \right) = 0
\end{align*}
\end{theorem}


\section{Lower and upper bounds}
\label{sectioni:NewLB}

This section begins with a simple lower and upper bound based on critical paths. 
Although this result may be known, we could not identify a specific reference in the literature.  

\begin{proposition}
\label{Prop:NewBounds}
In an instance with $N$ jobs, $M$ machines and processing times satisfying $a \leq t_{i,j} \leq b$, we have 
\begin{align*}
    a(N + M - 1) \leq OPT \leq C_{max}(\pi) \leq b(N + M - 1)
\end{align*}
\end{proposition}
\begin{proof}
Any critical path must pass through exactly $N + M - 1$ cells. Since each entry is between $a$ and $b$, the total sum of the path is at least $a(N + M - 1)$ and at most $b(N + M - 1)$. 
\end{proof}

This result allows us to estimate the number of distinct makespan values. 
Moreover, the upper bound $b(N + M - 1)$ will be useful in the asymptotic analysis in Section~\ref{section:AsymptoticBehavior}, which has implications for Taillard's conjecture. 
We refer to this upper bound as $UB$.

\begin{corollary}
\label{Coro:NumberMakespans}
In an instance with $N$ jobs, $M$ machines and integer processing times satisfying $a \leq t_{i,j} \leq b$, the number of distinct makespan values is at most $(N + M - 1)(b - a) + 1$. 
\end{corollary}

The estimate given by Corollary~\ref{Coro:NumberMakespans} can be compared with Heller's earlier bound~\cite{Heller}, which states that the number of makespans is at most $UB_H - LB_M + 1$. Since $LB_M \leq LB_M^+$, we refine this comparison by using $UB_H - LB_M^+ + 1$. 
Recall that $UB_H$ corresponds to the sum of all $NM$ processing times, while $LB_M^+$ is the sum of only $N + M - 1$ terms. This leads to the inequality: $(NM - (N +  M  - 1))a + 1 \leq UB_H - LB_M^+ + 1 \leq (NM - (N +  M  - 1))b + 1$. 
Some instances where Corollary~\ref{Coro:NumberMakespans} yields a better estimate than Heller's refined bound are captured in the next proposition. 

\begin{proposition}
Let the processing times be integers in the range $[a, b]$ with $1 \leq a < b$. For a fixed number of machines $M > \frac{b}{a}$, there is a finite number of instances where the new estimation can be worse, that is, where $(N + M - 1)(b - a) + 1 > UB_H - LB_M^+ + 1$
\end{proposition}
\begin{proof}
From the assumption $M > \frac{b}{a}$, it follows that $b - Ma < 0$. Let $N > \frac{-(M - 1)b}{b - Ma}$ be any number of jobs, then
\begin{align*}
N(b - Ma) &< -(M - 1)b\\
Nb + (M - 1)b &< NMa\\
(N + M  - 1)(b - a) &< (NM - (N + M - 1))a
\end{align*}
Since $(NM - (N + M - 1))a \leq UB_H - LB_M^+ $, the conclusion follows.  
\end{proof}

We now introduce a general framework for deriving lower bounds based on RD-paths and analyze how it relates to existing bounds. 

Let $\mathcal{P}$ be the set of all RD-paths of a matrix with $M$ rows and $N$ columns, and let $\mathcal{U} \subseteq \mathcal{P}$ be any non-empty subset. 
Using the matrix formulation of the PFSP and the fact that the maximum over a set is greater than or equal to the maximum over any non-empty subset, we observe (the minimization over $\pi$ range over all permutations of the jobs)
\begin{align}
    OPT &= \min_{\pi} \max_{P \in \mathcal{P}} s(\pi, P)  && \label{Equality:Makespan}\\
    &\geq \min_{\pi} \max_{P \in \mathcal{U}} s(\pi, P) && \label{Inequality:MinMax}\\
    &\geq \max_{P \in \mathcal{U}} \min_{\pi} s(\pi, P) && \label{Inequality:MaxMin}
\end{align}

For a single RD-path, the minimization over all permutations can be computed easily with a max-flow-min-cost network. 
However, for certain path structures, more efficient algorithms can be used,  
which is particularly relevant for inequality~(\ref{Inequality:MaxMin}), where the minimizations are done independently for each path. 
As an example, let $H_i$ denote the path that moves from $(1,1)$ down to $(i,1)$, then right to $(i,N)$, and then down to $(M,N)$, see Figure~\ref{fig:PathVisualization}. 
For any permutation $\pi$, the sum along $H_i$ is $s(\pi, H_i) = \sum_{i' < i} t_{i', \pi_1} + \sum_{j=1}^N t_{i,j} + \sum_{i' > i} t_{i', \pi_N}$.  
If we take $\mathcal{U} = \{H_i \mid 1 \leq i \leq M\}$, then inequality~\eqref{Inequality:MaxMin} leads to a bound similar to $LB_M^+$. Indeed, since the first and last jobs cannot be the same (when $N \geq 2$), we obtain a slight improvement
\begin{align*}
    LB_M^{++} = \max_{1 \leq i \leq M} \left\{ \sum_{j=1}^N t_{i,j} + \min_{a \neq b} \left\{ \sum_{i' < i} t_{i', a} + \sum_{i' > i} t_{i', b}  \right\} \right\}
\end{align*}

Similarly, the job-based bound $LB_J^+$ can be recovered with this method. 
This time, the set of paths $\mathcal{U}$ will be dynamic depending on the permutation $\pi$.  
Define $V_j$ to be the path that moves right from $(1,1)$ until reaching column $k$ corresponding to job $j$ (i.e., $\pi_k = j$), then down to row $M$, and finally right to $(M, N)$ (see Figure~\ref{fig:PathVisualization}). Let $\mathcal{U} = \{V_j \mid 1 \leq j \leq N\}$. Then inequality~\eqref{Inequality:MaxMin} yields $LB_J^+$.

\begin{figure}
    \centering
    \includegraphics[width=0.4\textwidth]{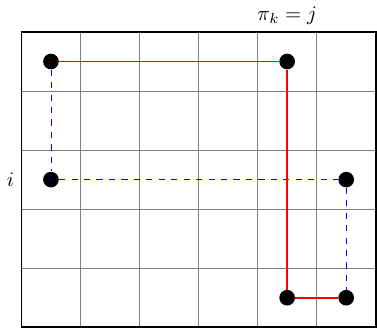}
  \caption{Visualization of path $H_i$ (dashed) and $V_j$ (solid)}
  \label{fig:PathVisualization}
  \end{figure}

We have shown that the popular bounds based on machines and jobs can be derived from inequality~(\ref{Inequality:MaxMin}), which is dominated by inequality~(\ref{Inequality:MinMax}) and may lead to some improvements. 
This possible improvement comes from the fact that the permutation that minimizes the sum of a path $P \in \mathcal{U}$ may increase the sum of another different path, and we should instead minimize simultaneously.

We now introduce a family of bounds, denoted $LB_{p,s}$, that are theoretically supported by solving inequality~(\ref{Inequality:MinMax}) for a given set of RD-paths $\mathcal{U}_{p,s}$. The detailed performance of $LB_{p,s}$ with the set of instances is given in Section~\ref{section:PerformanceLB}.
The key idea is to set a prefix of $p$ columns and a suffix of $s$ columns and define $\mathcal{U}_{p,s}$ to include all RD-paths that start at $(1,1)$ within the first $p$ columns, moves arbitrarily, then moves to the right until reaching column $N-s+1$, and in the last $s$ columns moves arbitrarily again to $(M, N)$, see Figure~\ref{fig:Ups}. 

\begin{figure}
    \centering
    \includegraphics[width=0.4\textwidth]{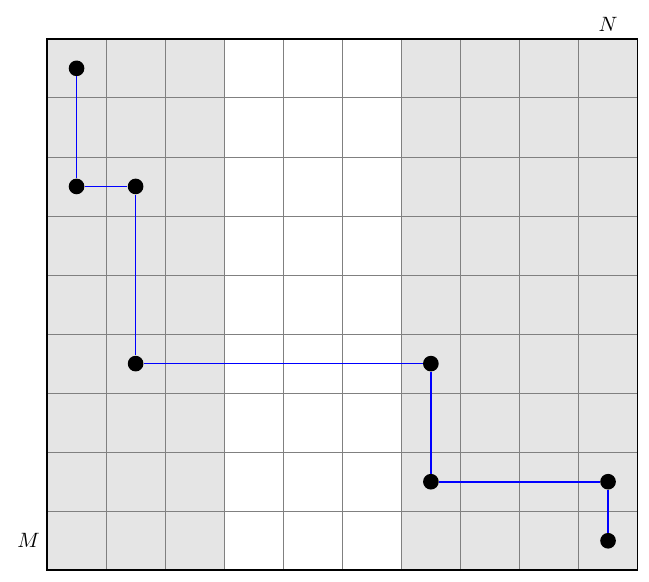}
  \caption{Visualization of path included in $U_{3, 4}$ for a matrix with $M = 9$ rows and $N = 10$ columns. The path is also included in $U_{2, 4}$}
  \label{fig:Ups}
\end{figure}

\begin{definition}
Let $p \geq 1$ and $s \geq 1$ be integers with $p + s \leq N$. We define the set $\mathcal{U}_{p,s}$ to contain all the RD-paths that move from $(1,1)$ to some $(i,p)$ via any RD-path, then move right to $(i, N-s+1)$, and finally follow any RD-path to $(M,N)$. We refer to such paths as $i$-type paths. 
\end{definition}

\begin{definition}
The \textit{prefix-suffix bound} $LB_{p,s}$ is defined as the solution to inequality~(\ref{Inequality:MinMax}) for the path set $\mathcal{U}_{p,s}$. 
\end{definition}

\begin{proposition}
\label{prop:CardinalityU}
The number of paths in $\mathcal{U}_{p,s}$ is $\binom{M + p + s - 2}{p + s - 1}$
\end{proposition}
\begin{proof}
By the definition of $\mathcal{U}_{p,s}$, any $i$-type path moves to the right from $(i,p)$ to $(i, N - s + 1)$. Thus, the cardinality of $\mathcal{U}_{p,s}$ coincides with the number of RD-paths in a matrix of $M$ rows and $p + s$ columns.
\end{proof}

Proposition~\ref{prop:CardinalityU} states that the cardinality of $\mathcal{U}_{p, s}$ depends only on the sum $p + s$ rather than on the specific values of $p$ and $s$. One may expect that by increasing the number of paths considered, the solution to expression~(\ref{Inequality:MinMax}) will yield a better lower bound. 
The following proposition formalizes this idea and states two good characteristics of a lower bound. First, that it is possible to obtain the optimum, and second, that if we fix $p$, then the lower bound $LB_{p, s}$ is a non-decreasing function of $s$. 
The same is true for a fixed $s$ and variable $p$. 
However, the behavior of $LB_{p, s}$ for fixed $p + s$ and different values of $p$ and $s$ is not clear. 

\begin{proposition}
\label{prop:BoundBehaviour}
The prefix-suffix bound satisfy 
\begin{enumerate}[i)]
    \item $LB_{p, N - p} = OPT$ for any $1 \leq p < N$
\item For $1 \leq p \leq P \leq N$ and $1 \leq s \leq S \leq N$ such that $P + S \leq N$, we have that $LB_{p,s} \leq LB_{P, S}$
\end{enumerate}
\end{proposition}
\begin{proof}
The set $\mathcal{U}_{p, N-p}$ is the same as the set of all RD-paths, and therefore $LB_{p, N - p} = OPT$. 
By definition, $\mathcal{U}_{p, s} \subseteq \mathcal{U}_{P, S}$, and therefore $LB_{p, s} \leq LB_{P, S}$. 
\end{proof}

\begin{proposition}
\label{prop:Complexity}
The prefix-suffix bound $LB_{p,s}$ can be computed in $O(N^{p + s}M(p + s))$ time
\end{proposition}
\begin{proof}
We enumerate all combinations of the first $p$ and last $s$ jobs in $O(N^{p+s})$. For each $i$-type path, the sum includes: (1) the maximum value of a path from $(1,1)$ to $(i,p)$, (2) the sum of processing times of unselected jobs in row $i$, and (3) the maximum path sum from $(i, N-s+1)$ to $(M,N)$. All of this can be computed for all $i$ using dynamic programming in $O(M(p + s))$ time. 
\end{proof}


\section{Asymptotic behavior}
\label{section:AsymptoticBehavior}

\begin{definition}
For $N, M \in \mathbb{N}$ and a random variable $\mathcal{T}$, let $\mathcal{I}(N, M, \mathcal{T})$ be the random instance with $N$ jobs and $M$ machines, 
where the processing times $T_{i,j}$ are independent and identically distributed random variables with the same distribution as $\mathcal{T}$.
\end{definition}

The following theorem analyzes the asymptotic behavior of the ratios $\frac{UB}{LB_M^+}$ and $\frac{UB}{LB_J^+}$ when either the number of jobs or the number of machines is fixed while the other increases. 
The result is used in the subsequent corollaries to advance both Taillard’s conjecture and the understanding of the asymptotic approximation ratio of any algorithm.
This approach is general and can be applied with any pair of lower and upper bounds, which may lead to better asymptotic ratios, especially in restricted instances where tighter bounds are available. 

\begin{theorem}
\label{Teo:AsymptoticRatio}
Let $N_0, M_0 \in \mathbb{N}$ and let $\mathcal{T}$ be a random variable with support in $[a,b]$ and mean $\mu$. Then 
\begin{align}
\lim_{N \to \infty} \mathbb{P} \left( \frac{UB}{LB_M^+}( \mathcal{I}(N, M_0, \mathcal{T}) ) \leq \frac{b}{\mu} \right) &= 1\\
\lim_{M \to \infty} \mathbb{P} \left( \frac{UB}{LB_J^+}( \mathcal{I}(N_0, M, \mathcal{T}) ) \leq \frac{b}{\mu} \right) &= 1
\end{align}
\end{theorem}
\begin{proof}
Note that $\sum_{j=1}^N T_{1,j} \leq LB_M^+$, hence:
\begin{align*}
    \frac{UB}{LB_M^+}( \mathcal{I}(N, M_0, \mathcal{T}) )  &\leq \frac{ (N + M_0 - 1)b }{ \sum_{j=1}^N T_{1,j} }\\
    &= \frac{ (N + M_0 - 1)b }{ N } \frac{ 1 }{ \frac{1}{N} \sum_{j=1}^N T_{1,j} }
\end{align*}
Clearly, $\lim_{N \to \infty} \frac{(N + M_0 - 1)b}{N} = b$ and by the weak law of large numbers (Theorem~\ref{Teo:WLLN}), we know that for any $\epsilon > 0$,
$\lim_{N \to \infty} \mathbb{P}(| \frac{1}{N} \sum_{j=1}^N T_{1,j} - \mu| > \epsilon) = 0$ and result~(\ref{AlgApprox:Ninf}) follows. A detailed derivation is given in Appendix~\ref{appendix:LimitAnalysis}. 

The case where $M$ tends to infinity is analogous, using the lower bound $\sum_{i=1}^M T_{i, 1} \leq LB_J^+$ with the following expression
\begin{align*}
   \frac{UB}{LB_J^+}( \mathcal{I}(N_0, M, \mathcal{T}) ) &\leq \frac{ (N_0 + M - 1)b }{ M } \frac{ 1 }{ \frac{1}{M} \sum_{i=1}^M T_{i, 1}  }
\end{align*}
\end{proof}

Taillard's conjecture can be approached by determining the smallest value $c$ such that $\lim_{N \to \infty} \mathbb{P}(cLB \geq OPT) = 1$. 
The conjecture claims that $c = 1$. 
Since $OPT \leq UB$, Theorem~\ref{Teo:AsymptoticRatio} implies the next corollary, 
establishing that $c \leq \frac{b}{\mu}$. 
In particular, for uniformly distributed processing times, we obtain $c \leq \frac{b}{\frac{1}{2}(a + b)} \leq 2$, which is especially relevant, as many commonly used benchmark instances are generated using the uniform distribution. 

\begin{corollary}
Let $N_0, M_0 \in \mathbb{N}$ and let $\mathcal{T}$ be a random variable with support in $[a,b]$ and mean $\mu$. Then
\begin{align*}
\lim_{N \to \infty} \mathbb{P}\left(\frac{b}{\mu} LB_M^+ (\mathcal{I}(N, M_0, \mathcal{T})) \geq OPT \right) &= 1\\
\lim_{M \to \infty} \mathbb{P}\left(\frac{b}{\mu} LB_J^+ (\mathcal{I}(N_0, M, \mathcal{T})) \geq OPT \right) &= 1\\
\end{align*}
\end{corollary}

Surprisingly, Theorem~\ref{Teo:AsymptoticRatio} also implies an asymptotic result for the approximation ratio of any algorithm, since the ratios in Theorem~\ref{Teo:AsymptoticRatio} provide upper bounds for the approximation ratio. 

\begin{corollary}
Let $N_0, M_0 \in \mathbb{N}$ and let $\mathcal{T}$ be a random variable with support in $[a,b]$ and mean $\mu$. Any algorithm $A$ satisfies 
\begin{align}
\lim_{N \to \infty} \mathbb{P} \left( \alpha_A( \mathcal{I}(N, M_0, \mathcal{T}) ) \leq \frac{b}{\mu} \right) &= 1 && \label{AlgApprox:Ninf}\\
\lim_{M \to \infty} \mathbb{P} \left( \alpha_A( \mathcal{I}(N_0, M, \mathcal{T}) ) \leq \frac{b}{\mu} \right) &= 1  && \label{AlgApprox:Minf}
\end{align}
\end{corollary}
\begin{proof}
By definition
\begin{align*}
    \alpha_A( \mathcal{I}(N, M_0, \mathcal{T}) ) &= \frac{A}{OPT}(\mathcal{I}(N, M_0, \mathcal{T}))\\
    &\leq \frac{UB}{LB_M^+}(\mathcal{I}(N, M_0, \mathcal{T}))
\end{align*}
and the result follows from Theorem~\ref{Teo:AsymptoticRatio}. The case where $M$ tends to infinity is analogous. 
\end{proof}


\section{Performance of $LB_{p,s}$}
\label{section:PerformanceLB}

In this section we present the performance of $LB_{p,s}$ with the well-known Taillard and VRF instances for different values of $p$ and $s$. 
Since the VRF instances are divided into a small group ($N \leq 60$) and a large group ($N \geq 100$), we applied the same criterion to partition the Taillard instances, enabling a consistent comparison for instances of comparable size. 
However, since $p + s$ appears as an exponent of $N$ in the time complexity (Proposition~\ref{prop:Complexity}), we restrict our experiments to $p + s \le 4$ for the small group and $p + s \leq 3$ for the large group. 
Additionally, we consider a lower bound \textit{Best}, which, for each instance selects the maximum among all $LB_{p,s}$. 
By Proposition~\ref{prop:BoundBehaviour}, this is equivalent to taking the maximum over all bounds with $p + s = 4$ for the small group and $p + s = 3$ for the large group. 
All lower bounds were implemented in C++ 
and executed on an Intel Xeon E5-2620 v2 2.1 GHz processor with 32 GB of RAM running on a Linux system. The code was compiled using g++ 11.4.0 with -O2 optimization. 
The source code and the lower bounds obtained in these experiments, including the execution times, are publicly available~\cite{SourceCode}. 

The results are compared with the best lower bounds provided by the official sources 
and with those presented in a recent study~\cite{LowerBoundsNew}, which, to the best of our knowledge, provides the state-of-the-art lower bounds for the Taillard instances and the small group of VRF instances. 

Table~\ref{Table:PerformanceLBTaillard} and Table~\ref{Table:PerformanceLBVRF} summarize the results for the Taillard and VRF instances, respectively.
For each lower bound, the tables report the number of instances in which it improved, matched, or worsened the previous lower bound. Additionally, we include the \textit{average relative percentage deviation (ARPD)}, which represents the mean of the \textit{relative percentage deviation (RPD)}, defined as $100 \dfrac{LB_{p,s}(I) - PLB(I)}{PLB(I)}$ for each instance $I$. Here, $PLB(I)$ denotes the previous best-known lower bound for instance $I$.

From the tables, we  observe the strong performance of $LB_{2, 2}$ with the small Taillard instances and $LB_{1, 2}$ with the large VRF instances. 
In the small VRF group, the best performing lower bound is $LB_{1,3}$, which improved most instances; however, the negative ARPD value indicates that, although it improved some instances, the magnitude of improvement was significantly smaller than the deterioration observed in others. 
Selecting the best value across all lower bounds significantly improves the ARPD in all the groups, suggesting that the bounds complement each other. Specifically, for each lower bound, there are instances where it performs better than the others and instances where another bound outperforms it. 

\begin{table}[]
\centering
\begin{tabular}{c|cccc|cccc|}
\multicolumn{9}{c}{Taillard} \\ 
\multirow{2}{*}{} & \multicolumn{4}{c|}{Small}                            & \multicolumn{4}{c|}{Large}                            \\ \hline
                  & $\uparrow$ & $\leftrightarrow$ & $\downarrow$ & ARPD & $\uparrow$ & $\leftrightarrow$ & $\downarrow$ & ARPD \\
$LB_{1,1}$ & 12 & 3 & 45 & -0.8807 & 21 & 1 & 38 & -0.1375\\
$LB_{2,1}$ & 39 & 2 & 19 & 0.1829 & 42 & 1 & 17 & 0.0763\\
$LB_{1,2}$ & 31 & 2 & 27 & -0.0815 & 47 & 2 & 11 & 0.1289\\
$LB_{3,1}$ & 47 & 2 & 11 & 0.8408 & - & - & - & -\\
$LB_{2,2}$ & 54 & 1 & 5 & 0.8952 & - & - & - & -\\
$LB_{1,3}$ & 35 & 2 & 23 & 0.4246 & - & - & - & -\\ \hline
Best & 55 & 1 & 4 & 1.2751 & 57 & 1 & 2 & 0.2493 \\ \hline
\end{tabular}
    \caption{Performance of the prefix-suffix lower bounds with the Taillard instances. The table shows, for each lower bound and instance group, the number of instances where the bound improved ($\uparrow$), remained equal ($\leftrightarrow$), or worsened ($\downarrow$), along with the average relative percentage deviation (ARPD)}
    \label{Table:PerformanceLBTaillard}
\end{table}

\begin{table}[]
\centering
\begin{tabular}{c|cccc|cccc|}
\multicolumn{9}{c}{VRF} \\ 
\multirow{2}{*}{} & \multicolumn{4}{c|}{Small}                            & \multicolumn{4}{c|}{Large}                            \\ \hline
                  & $\uparrow$ & $\leftrightarrow$ & $\downarrow$ & ARPD & $\uparrow$ & $\leftrightarrow$ & $\downarrow$ & ARPD \\
$LB_{1,1}$ & 29 & 0 & 211 & -4.3840 & 212 & 13 & 15 & 0.3132\\
$LB_{2,1}$ & 90 & 2 & 148 & -2.8163 & 237 & 3 & 0 & 0.6207\\
$LB_{1,2}$ & 140 & 1 & 99 & -1.9144 & 235 & 0 & 5 & 0.7035\\
$LB_{3,1}$ & 133 & 0 & 107 & -1.7209 & - & - & - & -\\
$LB_{2,2}$ & 177 & 0 & 63 & -0.7192 & - & - & - & -\\
$LB_{1,3}$ & 178 & 0 & 62 & -0.6193 & - & - & - & -\\ \hline
Best & 190 & 0  & 50 & -0.1392 & 240 & 0 & 0  & 0.7772 \\ \hline
\end{tabular}
    \caption{Performance of the prefix-suffix lower bounds with the VRF instances. The table shows, for each lower bound and instance group, the number of instances where the bound improved ($\uparrow$), remained equal ($\leftrightarrow$), or worsened ($\downarrow$), along with the average relative percentage deviation (ARPD)}
    \label{Table:PerformanceLBVRF}
\end{table}


\section{Concluding remarks}
\label{section:Conclusion}

This work presents a general scheme to derive lower bounds for the PFSP based on RD-paths. 
Using this framework, we proposed the prefix-suffix bound $LB_{p, s}$, which has a time complexity of $O(N^{p  + s}M(p + s))$ and improves several of the best-known lower bounds for the widely used Taillard and VRF benchmark instances.

We also introduced simple lower and upper bounds based on the maximum and minimum execution times of an instance.  
These were used to estimate the number of distinct makespan values, yielding tighter estimates in several cases compared to the previous estimation by Heller.   
Moreover, the upper bound is used to prove that for a fixed number of jobs or machines, the probability that $\frac{b}{\mu} LB_M^+ \geq OPT$ goes to one as the other parameter increases. 
This represents important progress on Taillard's conjecture, supporting its possible validity. 
Similarly, we proved that the probability that the approximation ratio of any algorithm is at most $\frac{b}{\mu}$ goes to one. 
For the case of a uniform distribution, commonly used in benchmark instance generation, this bound simplifies to $\frac{b}{\mu} \leq 2$. 

This research opens two interesting lines for future work. 
First, within the lower bound framework, it would be valuable to identify families of RD-paths that admit an efficient solution to expression~(\ref{Inequality:MinMax}). 
Second, further investigation is needed to prove or disprove Taillard's conjecture, or to extend the current asymptotic results to cases involving unbounded distributions with finite mean and variance.

\bibliographystyle{unsrtnat}
\bibliography{references}  

\appendix
\section{Limit analysis of Theorem~\ref{Teo:AsymptoticRatio}}
\label{appendix:LimitAnalysis}

Let $\epsilon > 0$ and $\lambda > 0$. We want to show that there exists $N_0$ such that for all $N > N_0$, the following holds $\mathbb{P} \left(\left| \frac{(N + M_0 - 1)b}{N} \frac{1}{ \frac{1}{N} \sum_j T_{1,j} } - \frac{b}{\mu} \right| > \epsilon \right) < \lambda$. 

Since $\lim_{N \to \infty} \frac{(N + M_0 - 1)b}{N} =  b$, for a given $\epsilon_b > 0$ there exists $N_b$ such that for all $N > N_b$ we have 
\begin{align}
    \left| \frac{(N + M_0 - 1)b}{N} - b \right| = \frac{(N + M_0 - 1)b}{N} - b < \epsilon_b && \label{b_bound}
\end{align}

Additionally, for a given $\epsilon_{\mu} > 0$, since $\lim_{N \to \infty} \mathbb{P} \left( \left| \frac{1}{N} \sum_j T_{1,j} - \mu  \right| > \epsilon_{\mu} \right) = 0$, there exists $N_{\mu}$ such that for all $N > N_{\mu}$ we have 
\begin{align}
\mathbb{P} \left( \left| \frac{1}{N} \sum_j T_{1,j} - \mu \right| > \epsilon_{\mu} \right) < \lambda && \label{mu_bound}
\end{align}

Let $N_0 = \max \{N_b, N_{\mu}\}$. Let $A$ be the event $\left| \frac{(N + M_0 - 1)b}{N} \frac{1}{ \frac{1}{N} \sum_j T_{1,j} } - \frac{b}{\mu} \right| > \epsilon$ and $B$ the event $ \left| \frac{1}{N} \sum_j T_{1,j} - \mu \right| \leq \epsilon_{\mu}$. 
For all $N > N_0$, we can write $\frac{(N + M_0 - 1)b}{N}$ as $b + \eta_b$ for some $\eta_b \in (0, \epsilon_b)$ and conditionally to event $B$ we can write $\frac{1}{N} \sum_j T_{1,j}$ as $\mu + \eta_\mu$ for some $\eta_\mu \in [-\epsilon_\mu, \epsilon_\mu]$. 
Therefore, using the law of total probability

\begin{align*}
    \mathbb{P} \left( A \right) &= \mathbb{P} \left( A \mid B \right) \mathbb{P}(B) + \mathbb{P} \left( A \mid B^c \right) \mathbb{P}(B^c)\\
    &\leq\mathbb{P} \left( A \mid B \right) \mathbb{P}(B)  + \mathbb{P}(B^c)\\
    &\leq  \mathbb{P} \left( A \mid B \right) \mathbb{P}(B) + \lambda \\
    &= \mathbb{P} \left( \left| \frac{b + \eta_b}{\mu + \eta_\mu} - \frac{b}{\mu} \right| > \epsilon \Big| B \right)\mathbb{P}(B) + \lambda
\end{align*}

Since we can set $\epsilon_b$ and $\epsilon_{\mu}$ arbitrarily small and this restricts the values $\eta_b$ and $\eta_\mu$, we can make $\mathbb{P} \left( \left| \frac{b + \eta_b}{\mu + \eta_\mu} - \frac{b}{\mu} \right| > \epsilon \Big| B \right) = 0$ and the result follows.

\end{document}